\documentclass[12pt, reqno]{amsart}

\usepackage[utf8]{inputenc}
\usepackage[T2A]{fontenc}
\usepackage[english]{babel}

\usepackage{hyperref}
\usepackage{enumitem}
\usepackage[margin=2.5cm]{geometry}

\usepackage{amssymb}
\usepackage{amscd}
\usepackage{graphicx}
\usepackage{xcolor}
\usepackage[all]{xy}

\hypersetup{
    hypertexnames=false,
    colorlinks,
    linkcolor={red!50!black},
    citecolor={blue!50!black},
    urlcolor={blue!80!black}
}

\newtheorem{theorem}[subsection]{Theorem}
\newtheorem{lemma}[subsection]{Lemma}
\newtheorem{sublemma}[subsubsection]{Lemma}

\newtheorem{corollary}[subsection]{Corollary}

\newtheorem{definition}[subsection]{Definition}

\newtheorem{remark}[subsection]{Remark}

\makeatletter
\@addtoreset{subsection}{section}
\@addtoreset{equation}{section}
\@addtoreset{figure}{section}
\@addtoreset{table}{section}
\makeatother

\makeatletter
\newcommand\testshape{family=\f@family; series=\f@series; shape=\f@shape.}
\def\myemphInternal#1{\if n\f@shape%
\begingroup\itshape #1\endgroup\/%
\else\begingroup\bfseries #1\endgroup%
\fi}
\def\myemph{\futurelet\testchar\MaybeOptArgmyemph}
\def\MaybeOptArgmyemph{\ifx[\testchar \let\next\OptArgmyemph
                 \else \let\next\NoOptArgmyemph \fi \next}
\def\OptArgmyemph[#1]#2{\index{#1}\myemphInternal{#2}}
\def\NoOptArgmyemph#1{\myemphInternal{#1}}
\makeatother

\newcommand{\bR}{\mathbb{R}}

\newcommand{\Int}{\mathop{\mathrm{Int}{}}\nolimits}
\newcommand{\Cl}[1]{\overline{#1}}

\newcommand\Ysp{Y}
\newcommand\Usp{U}
\newcommand\Vsp{V}

\newcommand\Sat[1]{S(#1)}

\newcommand\dif{h}
\newcommand{\Partition}{\Delta}

\newcommand{\leaf}{\omega}

\newcommand{\strip}{S}
\newcommand{\stripSurf}{Z}
\newcommand{\preStripSurf}{\stripSurf_0}
\newcommand{\bdX}{X}
\newcommand{\bdY}{Y}

\newcommand{\stInd}{{\lambda}}
\newcommand{\StInd}{\Lambda}

\newcommand{\bdGlueInd}{{\gamma}}
\newcommand{\BdGlueInd}{\Gamma}

 \newcommand{\qmap}{q}

\newcommand{\pr}{p} 

\renewcommand{\emptyset}{\varnothing}

\newcommand\singLeaves{\mathrm{Sing}(\Partition)}
\newcommand\specLeaves{\mathrm{Spec}(\Partition)}

\newcommand\hcl[1]{\mathrm{hcl}(#1)}
\newcommand\hclS[1]{\mathrm{hcl}_{S}(#1)}

\newcommand\typeInternal{{\rm(a)}}
\newcommand\typeBd{{\rm(b)}}

\newcommand\typeBdOne{{\rm(b1)}}
\newcommand\typeBdMany{{\rm(b2)}}

\newcommand\typeGlued{{\rm(c)}}
\newcommand\typeCycle{{\rm(c1)}}
\newcommand\typeReduce{{\rm(c2)}}

\newcommand\typeOneSideA{{\rm(c31)}}
\newcommand\typeOneSideB{{\rm(c32)}}
\newcommand\typeSpec{{\rm(c33)}}

\newcommand\Mman{\stripSurf}
\newcommand\tMman{\widetilde{\Mman}}
\newcommand\tPartition{\widetilde{\Partition}}
\newcommand\tleaf{\widetilde{\leaf}}

\newcommand\Upt[2]{U_{#1}^{#2}}
\newcommand\Uz[1]{\Upt{z}{#1}}
\newcommand\sign{\sigma}
\newcommand\ssign{s}

\newcommand\Vfam{\mathcal{V}}
\newcommand\Vpt[2]{\Vfam_{#1}^{#2}}
\newcommand\Vz[1]{\Vpt{z}{#1}}
\newcommand\hp{\hat{p}}

\newcommand\Nsp{N}
\newcommand\NFol{\mathcal{F}}
\newcommand\Fpt[1]{\mathcal{F}_{#1}}
\newcommand\Fpti[2]{\mathcal{G}_{#1,#2}}

\newcommand\eps{\varepsilon}

\begin{document}

\author{Sergiy Maksymenko, Eugene Polulyakh}
\email{maks@imath.kiev.ua, polulyah@imath.kiev.ua}
\address{Institute of Mathematics of NAS of Ukraine, Tereshchenkivska str. 3, Kyiv, 01004, Ukraine}

\title{Characterization of striped surfaces}

\begin{abstract}
Let $Z$ be a non-compact two-dimensional manifold and $\Delta$ be a one-dimensional foliation of $Z$ such that $\partial Z$ consists of leaves of $\Delta$ and each leaf of $\Delta$ is non-compact closed subset of $Z$.
We obtain a characterization of a subclass of such foliated surfaces $(Z,\Delta)$ glued from open strips $\mathbb{R}\times(0,1)$ with boundary leaves along some of their boundary intervals.
\end{abstract}


\maketitle

\section{Introduction}
Let $\stripSurf$ be a non-compact two-dimensional manifold and $\Partition$ be a one-dimensional foliation on $\stripSurf$ such that each leaf $\omega$ of $\Partition$ is homeomorphic to $\bR$ and is a closed subset of $\stripSurf$.

This kind of foliations was studied by W.~Kaplan~\cite{Kaplan:DJM:1940}, \cite{Kaplan:DJM:1941}, where he proved that every such foliation on $\bR^2$ has the following properties.
\begin{enumerate}[leftmargin=*, label=(\arabic*)]
\item\label{enum:Kaplan:level_sets}
There exists a pseudoharmonic function $f:\bR^2\to\bR$ taking constant values along leaves of $\Partition$ and ``strictly monotone in directions transversal to leaves'', see W.~Boothby~\cite{Boothby:AJM_1:1951}, \cite{Boothby:AJM_2:1951}, M.~Morse and J.~Jenkins~\cite{JenkinsMorse:AJM:1952}, M.~Morse~\cite{Morse:FM:1952} for further developments.

\item\label{enum:Kaplan:strips}
There exist at most countable family of leaves $\{\leaf_i\}_{i\in J}$ such that for every connected component $\strip$ of $\bR^2\setminus\{\leaf_i\}_{i\in J}$ one can find a homeomorphism $\phi:\strip\to\bR \times (0,1)$ sending the leaves in $\strip$ onto horizontal lines $\bR\times\{t\}$, $i\in(0,1)$.
\end{enumerate}
However the procedure of cutting along leaves $\leaf_i$ was not \textit{canonical}, as Kaplan tried to minimize the total number of strips, and for that reason the closures of connected components $\bR^2\setminus\{\leaf_i\}_{i\in J}$ can have a complicated structure.
In particular, the above homeomorphism $\phi$ does not always extend to an embedding of $\overline{\strip}$ into $\bR\times[0,1]$.

In~\cite{MaksymenkoPolulyakh:PGC:2015} the authors of the present paper introduced and studied a class of foliated surfaces $(\stripSurf,\Partition)$, called \myemph{striped}, glued from strips $S$ being open subsets of $\bR\times[0,1]$ and containing $\bR\times (0,1)$.

Further in~\cite[Theorem~1.8]{MaksymenkoPolulyakh:MFAT:2016} they also characterized a subclass of striped surfaces having the property that \textit{the quotient map $\pr:\stripSurf\to\stripSurf/\Partition$ into the set of leaves is a locally trivial fibration with fiber $\bR$} in terms of the so-called \myemph{special} leaves, see Definition~\ref{def:special_leaf}.
Such leaves are points where $\stripSurf/\Partition$ fails to be Hausdorff.
It was shown that under the above assumption a foliated surface $(\stripSurf,\Partition)$ admits ``striped structure'' if and only if the family of special leaves is locally finite, see Theorem~\ref{th:charact_stripedsurf_old} below.

In the present paper we introduce a more general notion of \myemph{singular} leaves, see Definition~\ref{def:regular_leaf}, corresponding to points of $\stripSurf/\Partition$ that do not have an open neighbourhood $\Usp$ such the pair $(\overline{\Usp},\Usp)$ is homeomoprhic with $\bigl( [0,1], (0,1) \bigr)$.

The aim of the present paper is to give a complete characterization of striped surfaces: we show that a foliated surface $(\stripSurf,\Partition)$ admits a ``striped structure'' if and only if the family of all singular leaves is locally finite, see Theorem~\ref{th:charact_stripedsurf}.

\subsection*{Structure of the paper}
In~\S\ref{sect:preliminaries} we recall necessary definitions of striped surfaces, types of leaves and relationships between them.
\S\ref{sect:cutting_along_isol_leaves} is devoted to proof of a technical result about cutting a foliated surface along isolated leaves, see Theorem~\ref{th:cutting}.
\S\ref{sect:main_results} contains main results of the paper: characterization of strips and striped surfaces, see Theorems~\ref{th:strip_charact} and~\ref{th:charact_stripedsurf}, based on Theorem~\ref{th:charact_stripedsurf_old_ext} being an extension of~\cite[Theorem~1.8]{MaksymenkoPolulyakh:MFAT:2016} and proved in~\S\ref{sect:proof:th:charact_stripedsurf_old_ext}.

\section{Preliminaries}\label{sect:preliminaries}
\subsection*{Space of leaves of a foliation}
A \myemph{foliated surface} is a pair $(\stripSurf,\Partition)$, where $\stripSurf$ is a two-dimensional manifold and $\Partition$ is a one-dimensional foliation on $\stripSurf$ such that each connected component of $\partial\stripSurf$ is a leaf of $\Partition$.

Denote by $\Ysp = \stripSurf/\Partition$ the set of all leaves of $\Partition$, and let $\pr:\stripSurf\to\stripSurf/\Partition$ be the natural projection associating to each $z\in\stripSurf$ the leaf of $\Partition$ containing $z$.
We will endow $\Ysp$ with the quotient topology, so a subset $\Usp\subset\Ysp$ is open if and only if $\pr^{-1}(\Usp)$ is open in $\stripSurf$.
It is well known, that then $\pr$ becomes an open map, see e.g.~\cite[Proposition~1.5]{Godbillon:F:1991} or~\cite[Theorem~4.10]{Tamura:Fol:ENG:1992}.

For a subset $\Usp\subset\stripSurf$ its \myemph{saturation}, $\Sat{\Usp}$, with respect to $\Partition$ is the union of all leaves of $\Partition$ intersecting $\Usp$.
Equivalently, $\Sat{\Usp}= \pr^{-1}(\pr(\Usp))$.
Notice that the openness of $\pr$ means that for each open $\Usp\subset\stripSurf$ its saturation $\Sat{\Usp}$ is open as well.
It easily follows from openness of $\pr$ that for each saturated subset $\Usp\subset\Ysp$ its closure $\overline{\Usp}$ is saturated as well.

If $\Usp$ is open and saturated, then by $\Partition_{\Usp}$ we will denote the induced foliation on $\Usp$ whose leaves are connected components of the intersections $\leaf\cap\Usp$ over all $\leaf\in\Partition$.

By the \myemph{Hausdorff closure}, $\hcl{y}$, of a point $y\in\Ysp$ we will mean the intersection of closures of all neighbourhoods of $y$, that is
\[
\hcl{y} = \bigcap_{\text{$\Vsp$ is a neighbourhood of $y$}} \overline{\Vsp}.
\]
Evidently, $y\in\hcl{y}$.
Moreover, $\Ysp$ is Hausdorff if and only if $\{y\}=\hcl{y}$ for each $y\in\Ysp$.

We will say that a point $y\in\Ysp$ is \myemph{special}%
\footnote{
See also~\cite[Definition~3]{HaefligerReeb:EM:1957} and \cite{GodbillonReeb:EM:1966} where such points are called \myemph{branch}.}
whenever $\{y\} \not=\hcl{y}$.

Similarly, for a leaf $\leaf\in\Partition$ let
\begin{align*}
\hcl{\leaf} &= \bigcap_{N(\leaf)} \overline{\Sat{N(\leaf)}}, &
\hclS{\leaf} &= \bigcap_{N_S(\leaf)} \overline{N_S(\leaf)},
\end{align*}
where $N(\leaf)$ runs over all open neighbourhoods of $\leaf$ and $N_S(\leaf)$ runs over all open saturated neighbourhoods of $\leaf$.
\begin{lemma}{\rm\cite[Lemma~3.5]{MaksymenkoPolulyakh:MFAT:2016}}
Let $\leaf\in\Partition$ and $y = \pr(\leaf)$.
Then
\begin{align*}
\hcl{\leaf} &= \hclS{\leaf} = \pr^{-1}(\hcl{y}), &
\pr(\hcl{\leaf}) &= \hcl{y}.
\end{align*}
\end{lemma}
This lemma is a consequence of openness of the projection $\pr$.
It also allows to give the following definition:

\begin{definition}\label{def:special_leaf}
A leaf $\leaf\in\Partition$ will be called \myemph{special}%
\footnote{
In~\cite{MaksymenkoPolulyakh:PGC:2015} authors introduced a class of ``striped'' foliated surfaces and used the term ``\myemph{special leaf}'' in a slightly distinct sense.
Further in~\cite{MaksymenkoPolulyakh:MFAT:2016}, \cite{MaksymenkoPolulyakh:PGC:2016}, and~\cite{MaksymenkoPolulyakhSoroka:PICG:2017} they classified a certain subclass of striped surfaces in terms of special leaves but in the sense of Definition~\ref{def:special_leaf}.
We will clarify the difference of definition in~\cite{MaksymenkoPolulyakh:PGC:2015} with Definition~\ref{def:special_leaf}, see Remark~\ref{rem:discussion_of_spec_leaves}.
}
whenever either of the following equivalent conditions hold:
\begin{itemize}
\item  $\leaf \not= \hcl{\leaf}$;
\item  $\leaf \not= \hclS{\leaf}$;
\item  $y=\pr(\leaf)$ is a special point of $\Ysp$, that is $y\not=\hcl{y}$.
\end{itemize}
\end{definition}

A homeomorphism $\dif:\stripSurf \to \stripSurf'$ between foliated surfaces $(\stripSurf,\Partition)$ and $(\stripSurf',\Partition')$ will be called \myemph{foliated} if for each leaf $\leaf\in\Partition$ its image, $\dif(\leaf)$, is a leaf of $\Partition'$.

\begin{definition}\label{def:local_crossections}
Fix any $a<b\in\bR$ and let $J=[a,b)$ or $J=(a,b)$.
Let also $\gamma:J \to \stripSurf$ be a continuous map such that $\gamma(J\cap\{a\})\in\partial\stripSurf$.

Then $\gamma$ is a \myemph{cross section} of $\Partition$, whenever $\pr\circ\gamma:J\to\stripSurf/\Partition$ is injective, that is for distinct $u,v\in J$ their images $\gamma(u)$ and $\gamma(v)$ belongs to distinct leaves of $\Partition$.
Also $\gamma$ is a \myemph{local cross section} of $\Partition$, whenever $\pr\circ\gamma:J\to\stripSurf/\Partition$ is locally injective.
\end{definition}

\begin{theorem}\label{th:cross_sections}
{\rm\cite[Theorem~2.8]{MaksymenkoPolulyakh:PGC:2016}}
Let $(\stripSurf,\Partition)$ be a connected foliated surface with countable base such that each leaf of $\Partition$ is non-compact and is also a closed subset of $\stripSurf$.
Suppose also that the family of all \myemph{special} leaves in the sense of Definition~\ref{def:special_leaf} is locally finite.
Then the following conditions are equivalent:
\begin{enumerate}[leftmargin=*, label={\rm(\Alph*)}]
\item\label{enum:th:charact_stripedsurf_old:eq:loc_triv_fibr}
the quotient map $\pr:\stripSurf\to\stripSurf/\Partition$ into the space of leaves if a locally trivial fibration with fiber $\bR$ and $\stripSurf/\Partition$ is locally homeomorphic with $[0,1)$ (though it is not necessarily a Hausdorff space);
\item\label{enum:th:charact_stripedsurf_old:eq:satur_open_nbh}
for each leaf $\leaf$ there exists an open saturated neighbourhood foliated homeomorphic with $\bR \times \Vsp$, where $\Vsp$ is an open subset of $[0,1)$;
\item\label{enum:th:charact_stripedsurf_old:eq:cross_sect}
each leaf of $\Partition$ admits a cross section.
\end{enumerate}
\end{theorem}

\begin{definition}\label{def:regular_leaf}
A leaf $\leaf \subset\Int{\stripSurf}$ will be called \myemph{regular} if there exists a saturated neighbourhood $\Usp$ of $\leaf$ such that the pair $(\Cl{\Usp},\Usp)$ is foliated homeomorphic with $\bigl(\bR\times[-1,1], \bR\times(-1,1))$ via a foliated homeomorphism sending $\leaf$ onto $\bR\times0$.

Similarly, a leaf $\leaf \subset\partial\stripSurf$ is \myemph{regular} if there exists a saturated neighbourhood $\Usp$ of $\leaf$ such that the pair $(\Cl{\Usp},\Usp)$ is foliated homeomorphic with $\bigl(\bR\times[0,1], \bR\times[0,1))$ via a foliated homeomorphism sending $\leaf$ onto $\bR\times0$.

A leaf being not regular will be called \myemph{singular}.
\end{definition}


Let $\specLeaves$ be the family of all special leaves of $\Partition$ and $\singLeaves$ be the family of all singular leaves.

\begin{lemma}\label{lm:props_of_regular_leaves}
Every regular leaf is non-special, that is every special leaf is singular, and so $\specLeaves \subset \singLeaves$.
\end{lemma}
\begin{proof}
Let $\leaf\subset\Int{\stripSurf}$ be a regular leaf belonging to the interior of $\stripSurf$, so there exists a saturated neighbourhood $\Usp$ and a foliated homeomorphism
\[ \phi:(\Cl{\Usp},\Usp) \to \bigl(\bR\times[-1,1], \bR\times(-1,1)\bigr)\]
such that $\phi(\leaf) = \bR\times 0$.
Then for each $t\in(0,1)$ the set $\Usp_{t} = \phi^{-1}\bigl(\bR\times(-t, t)\bigr)$ is an open foliated neighbourhood of $\leaf$ and $\Cl{\Usp_{t}} = \phi^{-1}\bigl(\bR\times[-t, t]\bigr)$.
Hence
\[
\hcl{\leaf} \subset \mathop{\cap}\limits_{t\in(0,1)} \Cl{\Usp_{t}}
    = \phi^{-1}\Bigl(\, \mathop{\cap}\limits_{t\in(0,1)} \bR\times[-t, t] \,\Bigr) = \phi^{-1}\bigl(\bR\times0\bigr) = \leaf,
\]
so $\leaf$ is non-special.

The case $\leaf\subset\partial\stripSurf$ is similar and we leave it for the reader.
\end{proof}

\subsection*{Strips}
A subset $\strip \subset \bR^2$ will be called a \myemph{strip} if there exist $u<v\in\bR$ such that
\begin{enumerate}[label=(\roman*), topsep=0pt]
\item\label{enum:strip:contains_open_strip}
$\bR \times (u,v) \ \subset \ \strip \ \subset \  \bR \times [u,v]$;
\item\label{enum:strip:open}
$\strip$ is open in the topology of $\bR \times [u,v]$.
\end{enumerate}
For such a strip we will use the following notation:
\begin{align*}
	\partial_{-}\strip &:= \strip \ \cap \ \bR \times \lbrace u \rbrace, &
	\partial_{+}\strip &:= \strip \ \cap \ \bR \times \lbrace v \rbrace, \\
	\partial\strip &:= \partial_{-}\strip \ \cup \ \partial_{+}\strip, &
	\Int\strip &:= \bR\times(u,v).
\end{align*}
Notice that the boundary $\partial\strip$ is open in $\bR\times\{u,v\}$ and therefore it is a disjoint union of at most countably many open (possibly unbounded) intervals.

Evidently, each strip $\strip$ possesses an oriented one-dimensional foliation into horizontal lines $\bR \times t$, $t \in (u,v)$, and boundary intervals of $\partial\strip$.
We will call that foliation \myemph{canonical}.

\subsection*{Striped atlas}
Let $\stripSurf$ be a two-dimensional topological manifold (surface) and $\preStripSurf = \bigsqcup \limits_{\stInd \in \StInd} \strip_{\stInd}$ be at most countable family of mutually disjoint strips.
A \myemph{striped atlas} on $\stripSurf$ is a map $\qmap: \preStripSurf \to \stripSurf$ such that
\begin{enumerate}[leftmargin=*, label=(\arabic*)]
\item
$\qmap$ is a \myemph{quotient} map, i.e. it is continuous, surjective, and a subset $\Usp\subset\stripSurf$ is open if and only if $\qmap^{-1}(\Usp) \cap \strip_{\stInd}$ is open in $\strip_{\stInd}$ for each $\stInd\in\StInd$;

\item
there exist two disjoint families $\mathcal{X} = \{\bdX_\bdGlueInd\}_{\bdGlueInd\in\BdGlueInd}$ and $\mathcal{Y} = \{\bdY_\bdGlueInd\}_{\bdGlueInd\in\BdGlueInd}$ of mutually distinct boundary intervals of $\preStripSurf$ enumerated by the same set of indexes $\BdGlueInd$ such that
\begin{enumerate}[leftmargin=*, label=(\alph*)]
\item
$\qmap$ is injective on $\preStripSurf \setminus (\mathcal{X} \cup \mathcal{Y})$;
\item
$\qmap(\bdX_\bdGlueInd) = \qmap(\bdY_\bdGlueInd)$ for each $\bdGlueInd\in\BdGlueInd$;
\item\label{enum:striped_atlas:XY:embeddings}
the restrictions $\qmap|_{\bdX_\bdGlueInd}: \bdX_\bdGlueInd \to \qmap(\bdX_\bdGlueInd)$ and
$\qmap|_{\bdY_\bdGlueInd}: \bdY_\bdGlueInd \to \qmap(\bdY_\bdGlueInd)$ are embeddings with closed images;
\end{enumerate}
\end{enumerate}

Notice that each striped atlas $\qmap$ induces on $\stripSurf$ a one-dimensional foliation obtained from canonical foliations on the corresponding strips $S_{\lambda}$.
We will call it the \myemph{canonical} foliation associated to the striped atlas $\qmap$ and denote by $\Partition$.
Evidently, each leaf of $\Partition$ is a homeomorphic image of $\bR$ and is also a closed subset of $\stripSurf$.

A foliated surface $(\stripSurf,\Partition)$ will be called \myemph{striped} whenever $\stripSurf$ has a striped atlas for which $\Partition$ is a canonical foliation.

Notice also that for each $\bdGlueInd \in \BdGlueInd$ we have a ``gluing'' homeomorphism
\begin{align}\label{equ:gluing_map}
\phi_{\bdGlueInd} &= \bigl(\qmap|_{\bdX_{\bdGlueInd}}\bigr)^{-1} \circ  \qmap|_{\bdY_{\bdGlueInd}}: \ \bdY_{\bdGlueInd} \to \bdX_{\bdGlueInd},
\end{align}
so a striped surface is obtained from a family of strips by gluing them along certain boundary intervals by homeomorphisms $\phi_{\bdGlueInd}$, see Figure~\ref{fig:type_of_leaves_gen}.

\graphicspath{{pictures/}}
\begin{figure}[ht]
\includegraphics[height=2cm]{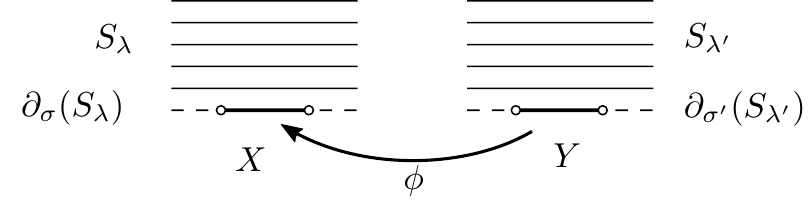}
\caption{Gluing boundary intervals}
\label{fig:type_of_leaves_gen}
\end{figure}

It is allowed to glue two strips along more than one pair of boundary components, and  one may also glue boundary components belonging to the same strip and even to \textit{the same side} of the same strip.

The latter possibility is the point of difference between the definition of \textit{special} leaves in~\cite{MaksymenkoPolulyakh:PGC:2015} and Definition~\ref{def:special_leaf}.

\subsection*{Standard foliated cylinder and M\"obius band}
Let $\strip=\bR\times[0,1]$, $\ssign=\pm1$, and $\phi_{\ssign}:\bR\times\{0\}\to\bR\times\{1\}$ be the homeomorphism given by $\phi_{\ssign}(x,0) = (\ssign x, 1)$.
Then the quotient mapping $\qmap: \strip \to \strip / \phi_{\ssign}$ is a striped atlas consisting of one strip.
The corresponding striped surface $\strip / \phi_{\ssign}$ will be called the \myemph{standard open cylinder} for $\ssign=+1$, and the \myemph{standard M\"obius band} for $\ssign=-1$.

\subsection*{Types of leaves of canonical foliation}
Let $\qmap: \bigsqcup \limits_{\stInd \in \StInd} \strip_{\stInd} \to \stripSurf$ be a striped atlas.
Then each leaf $\leaf$ of the associated canonical foliation $\Partition$ has precisely one of the following properties:

\begin{enumerate}[leftmargin=*, itemsep=1ex, label={\rm(\alph*)}]
\item[\typeInternal]
$\leaf = \qmap(\Int{\strip_{\stInd}})$ for some $\stInd\in\StInd$.

\item[\typeBd]
$\leaf \subset \qmap(\partial_{\sign}\strip_{\stInd}) \subset \partial\stripSurf$ for some $\stInd\in \StInd$ and $\sign\in\{-,+\}$.
This case splits into two subcases:

\begin{itemize}[itemsep=1ex]
\item[\typeBdOne]
$\leaf = \qmap(\partial_{\sign}\strip_{\bdGlueInd})$, so $\partial_{\sign}\strip_{\bdGlueInd}$ consists of a unique leaf;

\item[\typeBdMany]
$\leaf \subsetneq \qmap(\partial_{\sign}\strip_{\bdGlueInd})$, so $\partial_{\sign}\strip_{\bdGlueInd}$ contains more that one leaf.
\end{itemize}

\item[\typeGlued]
$\leaf = \qmap(\bdX_\bdGlueInd) = \qmap(\bdY_\bdGlueInd)$ for some $\bdGlueInd\in\BdGlueInd$, where $\bdX_\bdGlueInd \subset \partial_{\sign}\strip_{\stInd}$, $\bdY_\bdGlueInd \subset \partial_{\sign'}\strip_{\stInd'}$ for some
$\stInd,\stInd' \in \StInd$, and $\sign,\sign' \in \{-,+\}$.
This situation splits into the following three cases:

\begin{itemize}[itemsep=1ex]
\item[\typeCycle]
$\stInd = \stInd'$, $\bdX=\partial_{\sign}\strip_{\stInd}$, and $\bdY=\partial_{\sign'}\strip_{\stInd}$, so in this case $\sign'=-\sign$, that is we glue distinct sides of the same strip $\strip_{\stInd}$ and each of these sides consists of a unique interval;

\item[\typeReduce]
$\stInd \not= \stInd'$, $\bdX=\partial_{\sign}\strip_{\stInd}$, and $\bdY=\partial_{\sign'}\strip_{\stInd}$;

\item[\typeOneSideA]
$\stInd = \stInd'$, $\sign'=\sign$, and $\bdX \cup \bdY = \partial_{\sign}\strip_{\stInd}$;

\item[\typeOneSideB]
$\stInd = \stInd'$, $\sign'=\sign$, and $\bdX \cup \bdY \not= \partial_{\sign}\strip_{\stInd}$;

\item[\typeSpec]
all other cases.
\end{itemize}
\end{enumerate}

\begin{figure}[ht]
\includegraphics[height=2cm]{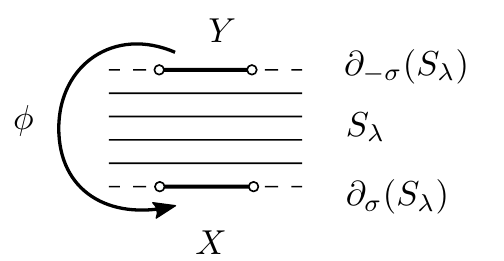}
\caption{Case \typeCycle {} ($\stInd' = \stInd$, $\sign' = -\sign$)}
\label{fig.type_of_leaves_c1}
\end{figure}

\begin{figure}[ht]
\includegraphics[height=2cm]{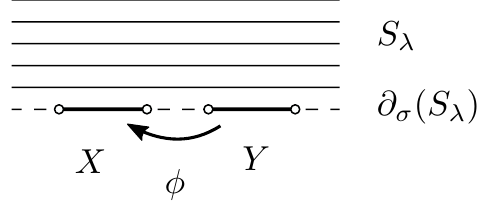}
\caption{Cases \typeOneSideA {} and \typeOneSideB {} ($\stInd' = \stInd$, $\sign' = \sign$)}
\label{fig.type_of_leaves_c31_and_c32}
\end{figure}

Thus the cases \typeOneSideA\ and \typeOneSideB\ correspond to gluing boundary intervals belonging to the same side of the same strip.

The following lemma characterizes special, regular, and singular leaves of canonical foliations of striped surfaces with types \typeInternal-\typeSpec.
In particular, it shows that the difference between singular and special leaves of the canonical foliation constitute leaves of type $\typeOneSideA$.
The proof is straightforward and we leave it for the reader.

\begin{lemma}\label{lm:relations_between_defs}
Let $\qmap: \bigsqcup \limits_{\stInd \in \StInd} \strip_{\stInd} \to \stripSurf$ be a striped atlas and $\Partition$ be a canonical foliation on $\stripSurf$.
Then the following statements hold.
\begin{enumerate}[leftmargin=*,label={\rm(\arabic*)}]
\item\label{enum:lm:relations_between_defs:special}
A leaf $\leaf\in\Partition$ is \myemph{special}, that is $\leaf\not=\hcl{\leaf}$, see Definition~\ref{def:special_leaf}, if and only if $\leaf$ is of one of the types \typeBdMany, \typeOneSideB, or \typeSpec.
\item\label{enum:lm:relations_between_defs:not_c31_c32}
The following conditions for a leaf $\leaf\in\Partition$ are equivalent:
\begin{enumerate}[leftmargin=*,label={\rm(\roman*)}, itemsep=1ex]
\item\label{enum:lm:relations_between_defs:not_c31_c32:cross_sect}
$\leaf\in\Partition$ \myemph{admits a cross section};
\item\label{enum:lm:relations_between_defs:not_c31_c32:satur_nbh}
$\leaf\in\Partition$ has an open saturated neighbourhood foliated homeomorphic with $\bR\times\Vsp$, where $\Vsp$ is an open subset of $[0,1)$;
\item\label{enum:lm:relations_between_defs:not_c31_c32:not_c31_c32}
$\leaf$ is not of types \typeOneSideA\ and \typeOneSideB.
\end{enumerate}

\item\label{enum:lm:relations_between_defs:regular}
A leaf $\leaf\in\Partition$ is \myemph{regular}, see Definition~\ref{def:regular_leaf}, if and only if $\leaf$ is of one of the types \typeInternal, \typeBdOne, \typeCycle, or \typeReduce.

\item\label{enum:lm:relations_between_defs:singular}
Correspondingly, a leaf $\leaf\in\Partition$ is \myemph{singular} if and only if $\leaf$ is of one of the types \typeBdMany, \typeOneSideA, \typeOneSideB, or \typeSpec.
\end{enumerate}
\end{lemma}

\begin{definition}\label{def:reduced_atlas}
An atlas $\qmap$ is called \myemph{reduced} if it does not contain leaves of types \typeCycle\ and \typeReduce.
\end{definition}

Denote $D = \qmap\bigl(\bigsqcup \limits_{\stInd \in \StInd} \partial\strip_{\stInd}\bigr)$.
Let also $\specLeaves$ be the family of all special leaves of $\Partition$ and $\singLeaves$ be the family of all singular leaves.

\begin{corollary}\label{cor:families_of_leaves}
{\rm c.f.~\cite[Lemma~7.2]{MaksymenkoPolulyakhSoroka:PICG:2017}}
The families $\specLeaves$, $\singLeaves$, $\partial\stripSurf$, and $D$ are locally finite, $\specLeaves \subset \singLeaves$, and $\partial \stripSurf \cup \specLeaves \subset D$.
Moreover, the atlas $\qmap$ is reduced if and only if $\partial \stripSurf \cup \specLeaves = D$.
\end{corollary}

\begin{remark}\label{rem:discussion_of_spec_leaves}\rm
In~\cite[Theorem~3.7]{MaksymenkoPolulyakh:PGC:2015} the authors proved the following statements.
\begin{enumerate}[leftmargin=*, label=(\arabic*)]
\item
If a leaf $\leaf\in\Partition$ is of type \typeCycle, then $\qmap(\strip_{\stInd})$ is a connected component of $\stripSurf$ foliated homeomorphic either with the standard cylinder or the standard M\"obius band.

\item\label{enum:reducing_strip}
If $\leaf$ is of type \typeReduce, then one can change the striped atlas $\qmap$ replacing strips $\strip_{\stInd}$ and $\strip_{\stInd'}$ with one strip obtained by gluing $\strip_{\stInd}$ and $\strip_{\stInd'}$ along $\bdX$ and $\bdY$.
This reduces the total number of strips in $\qmap$.

\item
It then follows from~\ref{enum:reducing_strip} that if $\stripSurf$ is connected and distinct from the standard cylinder and M\"obius band, then every striped atlas consisting of at most countably many strips can be replaced with a \myemph{reduced} one, i.e.\! having no leaves of types \typeCycle\ and \typeReduce.
In other words, in a reduced atlas each leaf of type \typeGlued\ is in fact of one of the types \typeOneSideA, \typeOneSideB, \typeSpec.
\end{enumerate}

However, the types \typeOneSideA, \typeOneSideB, \typeSpec\ were not distinguished in~\cite{MaksymenkoPolulyakh:PGC:2015} and a leaf having either of those types was called \myemph{special}.

On the other hand, if $(\stripSurf, \Partition)$ is a striped surface, then, due to \ref{enum:lm:relations_between_defs:special} of Lemma~\ref{lm:relations_between_defs}, a leaf $\leaf$ is \textit{special in the sense of Definition~\ref{def:special_leaf}} iff $\leaf$ is of one of the types \typeBdMany, \typeOneSideB, or \typeSpec.

Such an ambiguity led to an \textbf{incorrect formulation} of the definition of a reduced atlas saying that \textit{a striped atlas is called \textbf{reduced} whenever $D=\partial\stripSurf \cup \specLeaves$}, see a sentence before~\cite[Theorem~3.7]{MaksymenkoPolulyakhSoroka:PICG:2017}.
It \textbf{must be read} as Definition~\ref{def:reduced_atlas} or equivalently as
Corollary~\ref{cor:families_of_leaves}, that is $D=\partial\stripSurf \cup \singLeaves$.
Then~\cite[Theorems 7.3 \& 8.1]{MaksymenkoPolulyakhSoroka:PICG:2017} remain true.
\end{remark}

\begin{lemma}\label{lm:rel_between_leaf_types}
Suppose each leaf of $\Partition$ admits a cross section, so it satisfies each of the equivalent  conditions~\ref{enum:th:charact_stripedsurf_old:eq:loc_triv_fibr}-\ref{enum:th:charact_stripedsurf_old:eq:cross_sect} of Theorem~\ref{th:cross_sections}.
If the atlas $\qmap$ is reduced, then for a leaf $\leaf \subset\Int{\stripSurf}$ the following conditions are equivalent:
\begin{enumerate}[label={\rm(\roman*)}, topsep=0pt]
\item\label{enum:cor:rel_between_leaf_types:c33}
$\leaf$ is of type \typeSpec;
\item\label{enum:cor:rel_between_leaf_types:spec}
$\leaf$ is special, i.e. $\leaf\not=\hcl{\leaf}$;
\item\label{enum:cor:rel_between_leaf_types:sing}
$\leaf$ is singular;
\item\label{enum:cor:rel_between_leaf_types:c}
$\leaf$ is of type \typeGlued, so $\leaf = \qmap(\bdX_\bdGlueInd) = \qmap(\bdY_\bdGlueInd)$ for some $\bdGlueInd\in\BdGlueInd$.
\end{enumerate}
If, in addition, each singular leaf is contained in $\partial\stripSurf$, then there is no leaves of type \typeGlued, whence  $\qmap$ is a homeomorphism, and so $\stripSurf$ is a disjoint union of strips.
\end{lemma}
\begin{proof}
The implications
\ref{enum:cor:rel_between_leaf_types:c33}$\Rightarrow$%
\ref{enum:cor:rel_between_leaf_types:spec}$\Rightarrow$%
\ref{enum:cor:rel_between_leaf_types:sing} directly follow from
\ref{enum:lm:relations_between_defs:special} and \ref{enum:lm:relations_between_defs:singular} of Lemma~\ref{lm:relations_between_defs}.

\ref{enum:cor:rel_between_leaf_types:sing}$\Rightarrow$\ref{enum:cor:rel_between_leaf_types:c}.
If $\leaf$ is a singular leaf contained in $\Int{\stripSurf}$, then by \ref{enum:lm:relations_between_defs:singular} of Lemma~\ref{lm:relations_between_defs} $\leaf$ is of one of the types \typeOneSideA, \typeOneSideB, or \typeSpec, and so it is of type \typeGlued.

\ref{enum:cor:rel_between_leaf_types:c}$\Rightarrow$\ref{enum:cor:rel_between_leaf_types:c33}.
Suppose a leaf $\leaf$ is of type \typeGlued.
Since the atlas is reduced, $\leaf$ is not of types \typeCycle\ and \typeReduce.
Moreover, as $\leaf$ admits a cross section, if follows from~\ref{enum:lm:relations_between_defs:not_c31_c32:cross_sect} of Lemma~\ref{lm:relations_between_defs} that $\leaf$ is not of types \typeOneSideA\ and \typeOneSideB\ as well.
Hence $\leaf$ is of type \typeSpec.

For the last statement notice that by the construction of the atlas each leaf of type \typeGlued\ is contained in $\Int{\stripSurf}$.
Moreover, the equivalence \ref{enum:cor:rel_between_leaf_types:sing}$\Leftrightarrow$\ref{enum:cor:rel_between_leaf_types:c} implies that every such leaf must also be singular.
Hence if each singular leaf $\leaf$ of $\Partition$ is contained in $\partial\stripSurf$, then $\qmap$ will have no leaves of type \typeGlued.
In other words, no strips are glued via $\qmap$, whence $\qmap$ is a homeomorphism and $\stripSurf$ is a disjoint union of strips.
\end{proof}

\section{Cutting foliated surface along isolated leaves}\label{sect:cutting_along_isol_leaves}

\begin{definition}\label{def:isolated_leaf}
Let $(\Mman,\Partition)$ be a foliated surface.
A leaf $\leaf$ is called \myemph{isolated}, if for each $z\in\leaf$ there exists a foliated chart that contains this point and intersects $\leaf$ by an arc.
In other words, there exist an open neighbourhood $W$ of $z$ and an imbedding $\phi:(-1,1)\times(-1,1)\to\Mman$ such that
\begin{itemize}[leftmargin=2em]
\item $\phi\bigl((-1,1)\times(-1,1)\bigr) = W$,
\item $\phi^{-1}(\leaf) = (-1,1)\times 0$,
\item $\phi(-1,1) \times t$ is contained in some leaf of $\Partition$ for each $t\in(-1,1)$.
\end{itemize}
\end{definition}

\begin{theorem}\label{th:cutting}
Let $(\Mman,\Partition)$ be a foliated surface and $\singLeaves \subset\Int{\Mman}$ be a locally finite family of isolated leaves.
Then there exists a foliated surface $(\tMman, \tPartition)$ and a continuous map $p:\tMman\to\Mman$ having the following properties.
\begin{enumerate}[leftmargin=*, label={\rm(\arabic*)}]
\item\label{enum:th:cutting:p_is_quotient}
$p$ is a quotient map, so a subset $A\subset \Mman$ is open if and only if $p^{-1}(A)$ is open in $\tMman$;
\item\label{enum:th:cutting:p_M_Sigma_homeo}
the restriction $p:\tMman\setminus p^{-1}(\singLeaves) \to \Mman\setminus\singLeaves$ is a foliated homeomorphism;
\item\label{enum:th:cutting:p_inv_Sigma_2to1}
for each leaf $\leaf \in \singLeaves$ the inverse image $p^{-1}(\leaf)$ consists of two leaves $\tleaf_1, \tleaf_2\subset\partial\tMman$ of $\tPartition$ such that $p|_{\tleaf_i}: \tleaf_i \to \leaf$, $i=1,2$, is a homeomorphism.
\end{enumerate}
\end{theorem}
\begin{proof}
Put $\Usp = \Mman\setminus\singLeaves$ and let $j:\Usp\subset\Mman$ be the inclusion map.
For each leaf $\leaf\in\singLeaves$ and each point $z\in\leaf$ let also $\Uz{} = (-1,1)\times(-1,1)$,
\begin{align*}
\Uz{-} &= (-1,1)\times(-1,0], &
\Uz{+} &= (-1,1)\times[0,1),
\end{align*}
and $\phi_z: \Uz{} \to \Mman$ be an embedding guaranteed by Definition~\ref{def:isolated_leaf}, so
\begin{itemize}
\item $\phi_z(\Uz{})$ is open in $\Mman$,
\item $\phi_z^{-1}(\leaf) = (-1,1)\times 0$,
\item $\phi_z(-1,1) \times t$ is contained in some leaf of $\Partition$.
\end{itemize}
Since $\singLeaves$ is locally finite, one can additionally assume that
\begin{itemize}
\item $\phi_z^{-1}(\singLeaves) = \phi_z^{-1}(\leaf) = (-1,1)\times 0$.
\end{itemize}
Let
\[
\Nsp \ = \ \Usp \ \bigsqcup \
 \mathop{\sqcup}\limits_{\substack{ z\in\leaf\in\singLeaves, \\ \sign=\pm}} \Uz{\sign}
\]
be the disjoint union of $\Usp$ with sets $\Uz{-}$ and $\Uz{+}$ over all leaves $\leaf\in\singLeaves$ and $z\in\leaf$.
Then we have a natural map $\hp:\Nsp\to\Mman$ defined by
\begin{align*}
\hp|_{\Usp} &= j:\Usp\,\subset\,\Mman, \\
\hp|_{\Uz{\sign}}& = \phi_z|_{\Uz{\sign}} = \phi_{z, \sign}: \Uz{\sign} \to \Mman, \ (z\in\singLeaves, \sign=\pm).
\end{align*}

\begin{sublemma}\label{lm:hp_is_quotient}
The map $\hp$ is a quotient map.
\end{sublemma}
\begin{proof}
Since $\hp$ is continuous and surjective and $\hp|_{\Usp}$ and $\hp|_{\Uz{\sign}}$ are embeddings, one should only check that for a subset $A\subset \Mman$ the following conditions are equivalent:
\begin{enumerate}[label=(\alph*)]
\item\label{enum:lm:hp_is_open:A_open}
$A$ is open;
\item\label{enum:lm:hp_is_open:A_U}
$A \cap \hp(\Usp)$ is open in $\hp(\Usp)$, and $A \cap \phi_{z, \sign}(\Uz{\sign})$ is open in $\phi_{z, \sign}(\Uz{\sign})$ for all $z\in\singLeaves$ and $\sign=\pm$.
\end{enumerate}

The implication \ref{enum:lm:hp_is_open:A_open}$\Rightarrow$\ref{enum:lm:hp_is_open:A_U} is evident.

\ref{enum:lm:hp_is_open:A_U}$\Rightarrow$\ref{enum:lm:hp_is_open:A_open}
Evidently, the sets $\phi_{z, -}(\Uz{-})$ and $\phi_{z, +}(\Uz{+})$ form a finite closed cover of $\phi_z(\Uz{})$ for an arbitrary $z \in \leaf \in \singLeaves$. Therefore, if both intersections $A \cap \phi_{z, -}(\Uz{-})$ and $A \cap \phi_{z, +}(\Uz{+})$ are open respectively in $\phi_{z, -}(\Uz{-})$ and $\phi_{z, +}(\Uz{+})$, then $A \cap \phi_z(\Uz{})$ is open in $\Uz{}$.

Since the sets $\hp(\Usp)$ and $\phi_z(\Uz{})$, $z \in \leaf \in \singLeaves$, are open in $\Mman$, and so are their intersections with $A$, it follows that $A$ is open in $\Mman$ as well.
%
%
\end{proof}

We will now represent $\hp$ as a composition of two continuous maps
\begin{equation}\label{equ:hp_p_q}
\hp = p \circ q: \Nsp \xrightarrow{~~q~~} \tMman \xrightarrow{~~p~~} \Mman,
\end{equation}
where $p$ will satisfy the statement of Theorem~\ref{th:cutting}.

For every $z\in\singLeaves$ let
\[
\Vz{} = \{ \Upt{x}{\sign} \mid z \in \phi_x(\Upt{x}{\sign}), \ x \in \singLeaves, \ \sign=\pm \}
\]
be the family of all $\Upt{x}{\sign}$ whose image in $\Mman$ contains $z$.
Notice that for each $\Upt{x}{\sign} \in \Vz{}$ there exists $\epsilon>0$ such that exactly one of the following two conditions holds:
\begin{align*}
&\text{either}& &\phi_x^{-1}\Bigl(\phi_z\bigl( 0\times [0,\epsilon] \bigr)\Bigr) \subset \Upt{x}{\sign}, &
&\text{or} & \phi_x^{-1}\Bigl(\phi_z\bigl( 0\times [-\epsilon,0] \bigr)\Bigr) \subset \Upt{x}{\sign}.
\end{align*}
Hence $\Vz{}$ is a disjoint union of two subfamilies, see Fig.~\ref{fig:V_z}:
\begin{align*}
\Vz{-} &= \left\{ \Upt{x}{\sign} \in \Vz{} \mid \phi_x^{-1}\Bigl(\phi_z\bigl( 0\times [-\epsilon,0] \bigr)\Bigr) \subset \Upt{x}{\sign} \ \text{for some} \ \epsilon>0 \right\}, \\
\Vz{+} &= \left\{ \Upt{x}{\sign} \in \Vz{} \mid \phi_x^{-1}\Bigl(\phi_z\bigl( 0\times [0,\epsilon] \bigr)\Bigr) \subset \Upt{x}{\sign} \ \text{for some} \ \epsilon>0 \right\}.
\end{align*}
Moreover, we get the following partition of $\Nsp$:
\[
\NFol = \{ \Fpt{x} \}_{x\in\Usp} \ \bigsqcup \
\{ \Fpti{x}{\sign} \}_{x\in\singLeaves, \sign=\pm},
\]
where
\begin{align*}
\Fpt{x}&= j^{-1}(x) \bigcup \mathop\cup\limits_{z\in\singLeaves, \sign=\pm} \phi_{z,\sign}^{-1}(x), &
&(x\in\Usp), \\
\Fpti{x}{\sign} &= \{ \phi_{z, \nu}^{-1}(x) \mid \Upt{z}{\nu} \in \Vpt{x}{\sign},\ z \in \singLeaves,\ \nu = \pm \}, &
&(x\in\singLeaves, \ \sign=\pm).
\end{align*}
\graphicspath{{pictures/}}
\begin{figure}[ht]
\includegraphics[height=4cm]{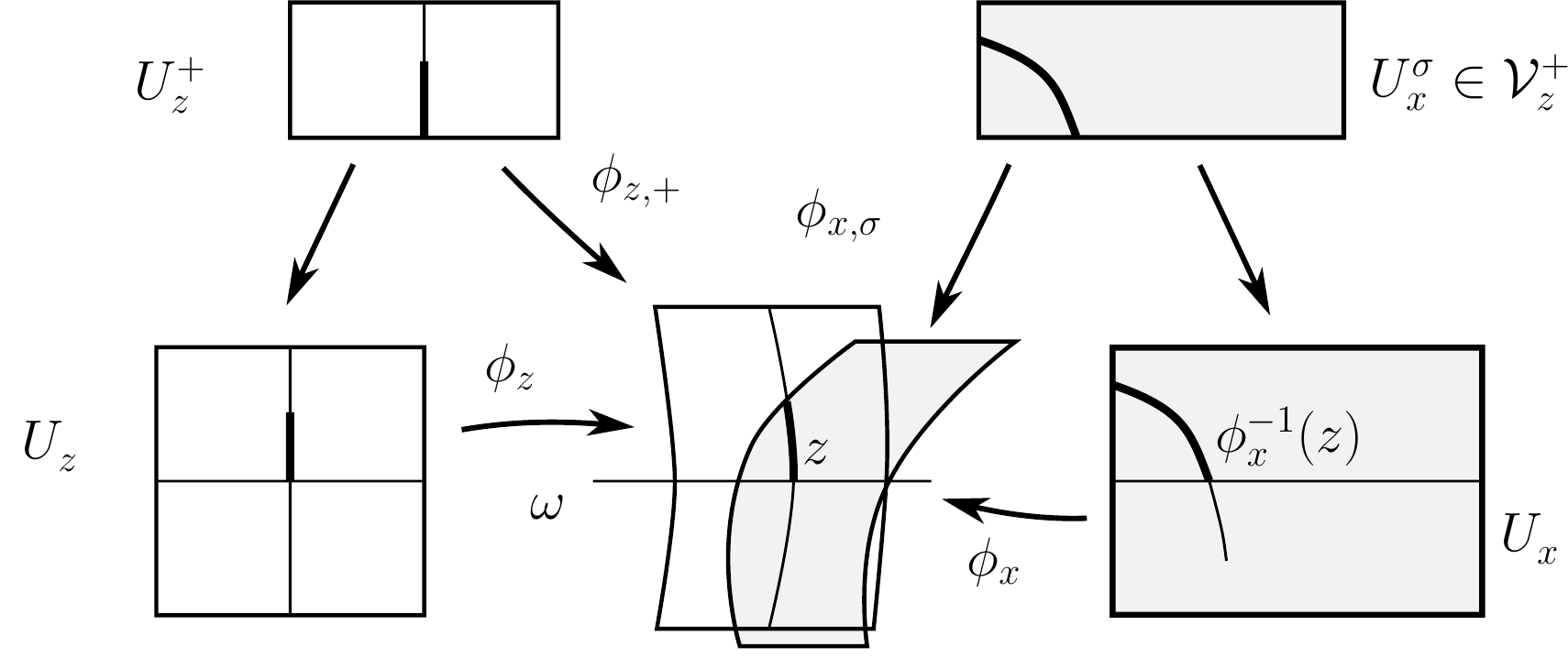}
\caption{}
\label{fig:V_z}
\end{figure}

Let $\tMman = \Nsp/\NFol$ be the set of elements of $\NFol$ and $q:\Nsp\to\tMman$ be the quotient map.
Endow $\tMman$ with the corresponding quotient topology: so a subset $A\subset \tMman$ is open if and only if its inverse $q^{-1}(A)$ is open in $\Nsp$.
It easily follows that $q$ is an open map.

Notice that
\begin{align*}
\hp(\Fpt{x}) &= x, \ (x\in\Usp), &
\hp(\Fpti{x}{\sign}) &= x, \ (x\in\singLeaves, \ \sign=\pm),
\end{align*}
whence $\hp$ induces a map $p:\tMman\to\Mman$ giving the required decomposition $\hp = p\circ q$, see~\eqref{equ:hp_p_q}.

Verification of properties~\ref{enum:th:cutting:p_is_quotient}-\ref{enum:th:cutting:p_inv_Sigma_2to1} is left for the reader.
\end{proof}

\section{Main results}\label{sect:main_results}
In this section we will assume that \textit{$(\stripSurf,\Partition)$ is a foliated surface with countable base and such that each leaf of $\Partition$ is homeomorphic to $\bR$ and is a closed subset of $\stripSurf$}.
Let also $\specLeaves\subset\singLeaves$ be the families of all special and singular leaves of $\Partition$ respectively.

The following statement characterizes striped surfaces without leaves of types \typeOneSideA\ and \typeOneSideB.
\begin{theorem}\label{th:charact_stripedsurf_old}
{\rm\cite[Theorem~1.8]{MaksymenkoPolulyakh:MFAT:2016},
c.f.\! also~\cite[Theorem~7.4]{MaksymenkoPolulyakhSoroka:PICG:2017}.}
The following conditions are equivalent:
\begin{enumerate}[leftmargin=*, label={\rm(\arabic*)}, itemsep=1ex]
\item\label{enum:th:charact_stripedsurf_old:str_atlas}
$(\stripSurf,\Partition)$ admits a striped atlas without leaves of types \typeOneSideA\ and \typeOneSideB;

\item\label{enum:th:charact_stripedsurf_old:locfin}
$\specLeaves$ is a locally finite family and $\Partition$ satisfies each of the equivalent  conditions~\ref{enum:th:charact_stripedsurf_old:eq:loc_triv_fibr}-\ref{enum:th:charact_stripedsurf_old:eq:cross_sect} of Theorem~\ref{th:cross_sections}.
\end{enumerate}
\end{theorem}
\begin{proof}
\ref{enum:th:charact_stripedsurf_old:str_atlas}$\Rightarrow$\ref{enum:th:charact_stripedsurf_old:locfin}.
By Corollary~\ref{cor:families_of_leaves} $\specLeaves$ is locally finite.
Also as $\qmap$ has no leaves of types \typeOneSideA\ and \typeOneSideB, we get from \ref{enum:lm:relations_between_defs:not_c31_c32} of Lemma~\ref{lm:relations_between_defs} that each leaf of $\Partition$ admits a cross section, i.e.\! condition~\ref{enum:th:charact_stripedsurf_old:eq:cross_sect} of Theorem~\ref{th:cross_sections} holds.

The implication \ref{enum:th:charact_stripedsurf_old:locfin}$\Rightarrow$\ref{enum:th:charact_stripedsurf_old:str_atlas}
is established in~\cite[Theorem~1.8]{MaksymenkoPolulyakh:MFAT:2016}.
\end{proof}

The following extension of Theorem~\ref{th:charact_stripedsurf_old} allows to check the existence of cross sections only for leaves in the interior of $\stripSurf$.

\begin{theorem}\label{th:charact_stripedsurf_old_ext}
Suppose that $\singLeaves$ is locally finite, and each leaf of $\Partition$ contained in $\Int{\stripSurf}$ admits a cross section.
Then each leaf in $\partial\stripSurf$ also admits a cross section.
Hence by Theorem~\ref{th:charact_stripedsurf_old}  $(\stripSurf,\Partition)$ is a striped surface.
\end{theorem}

We will prove it in \S\ref{sect:proof:th:charact_stripedsurf_old_ext}.
As a consequence of Theorem~\ref{th:charact_stripedsurf_old_ext} we get the following characterization of strips.

\begin{theorem}\label{th:strip_charact}
Suppose $\stripSurf$ is connected, the family $\singLeaves$ is locally finite, and $\singLeaves\subset \partial\stripSurf$.
Then $\stripSurf$ is foliated homeomorphic either to a standard cylinder or to a standard M\"obius band or to a strip.
\end{theorem}

\begin{proof}
Suppose $(\stripSurf,\Partition)$ is neither a standard cylinder nor a standard M\"obius band.
We should show that then it is foliated homeomorphic to a strip.

By Lemma~\ref{lm:props_of_regular_leaves} $\specLeaves \subset \singLeaves$, whence $\specLeaves$ is locally finite as well.
Moreover, as $\specLeaves \subset \singLeaves \subset \partial\stripSurf$, it follows that each leaf in $\Int{\stripSurf}$ is regular, and therefore it admits a cross section.
Hence by Theorem~\ref{th:charact_stripedsurf_old_ext}  $(\stripSurf,\Partition)$ admits a reduced atlas $\qmap$.
As $\singLeaves \subset \partial\stripSurf$, if follows from Lemma~\ref{lm:rel_between_leaf_types} that $\stripSurf$ is a disjoint union of strips.
But $\stripSurf$ is connected, so it is a strip itself.
\end{proof}

The next statement characterizes all striped surfaces.

\begin{theorem}\label{th:charact_stripedsurf}
The following conditions are equivalent:
\begin{enumerate}[label={\rm(\arabic*)}, itemsep=1ex]
\item\label{enum:th:charact_stripedsurf:str_atlas}
$(\stripSurf,\Partition)$ admits a striped atlas;

\item\label{enum:th:charact_stripedsurf:sing_leaf_loc_finite}
the family $\singLeaves$ of all \myemph{singular} leaves is locally finite.
\end{enumerate}
\end{theorem}
\begin{proof}
The implication \ref{enum:th:charact_stripedsurf:str_atlas}$\Rightarrow$\ref{enum:th:charact_stripedsurf:sing_leaf_loc_finite} is contained in Corollary~\ref{cor:families_of_leaves}.

\ref{enum:th:charact_stripedsurf:sing_leaf_loc_finite}$\Rightarrow$\ref{enum:th:charact_stripedsurf:str_atlas}.
Suppose the family $\singLeaves$ of singular leaves is locally finite.
By assumption each leaf $\leaf$ of $\Partition$ is homeomorphic to $\bR$ and is a closed subset of $\stripSurf$.
Therefore every foliated chart intersects $\leaf$ by a discrete family of arcs.
Hence each singular leaf is isolated, and therefore by Theorem~\ref{th:cutting} there exists a foliated surface $(\tMman, \tPartition)$ and a quotient map $\qmap:\tMman\to\Mman$ such that
\begin{enumerate}[leftmargin=*, label=\alph*), itemsep=1ex]
\item
the restriction $\qmap:\tMman\setminus \qmap^{-1}(\singLeaves) \to \stripSurf\setminus\singLeaves$ is a foliated homeomorphism, and
\item\label{enum:cutting}
for each leaf $\leaf \in \singLeaves$ the inverse image $\qmap^{-1}(\leaf)$ consists of two leaves $\tleaf_1,\tleaf_2\subset\partial\tMman$ of $\tPartition$ such that $\qmap|_{\tleaf_i}: \tleaf_i \to \leaf$, $i=1,2$, is a homeomorphism.
\end{enumerate}

Let $\mathrm{Sing}(\tPartition)$ be the family of all singular leaves of $\tPartition$.
Since $\singLeaves$ is locally finite, and $p$ is two-to-one on $\mathrm{Sing}(\tPartition)$, it follows that $\mathrm{Sing}(\tPartition)$ is locally finite as well.
Moreover, $\mathrm{Sing}(\tPartition)\subset\partial\tMman$, whence by Theorem~\ref{th:strip_charact} every connected component of $\tMman$ is a strip.
Hence $\qmap$ is a striped atlas for $(\stripSurf,\Partition)$.
\end{proof}

\section{Proof of Theorem~\ref{th:charact_stripedsurf_old_ext}}
\label{sect:proof:th:charact_stripedsurf_old_ext}
\begin{lemma}\label{lm:cross_section_meets_a_singular_leaf}
Let $(\stripSurf,\Partition)$ be a connected striped surface with countable base and $\partial\stripSurf= \emptyset$.
Let also $\delta: [0, 1] \to \stripSurf$ be a local cross section such that the points $\delta(0)$ and $\delta(1)$ belong to the same leaf of $\Partition$.
Then exactly one of the following two conditions holds:
\begin{enumerate}[leftmargin=*,label={\rm(\alph*)}]
\item\label{enum:lm:cross_section_meets_a_singular_leaf:cylinder}
either $\stripSurf$ is a standard cylinder or a standard M\"obius band and it coincides with the saturation $\Sat{\delta([0, 1])}$ of the image of $\delta$;
\item\label{enum:lm:cross_section_meets_a_singular_leaf:intersect_sing_leaf}
or $\delta$ intersects a singular leaf $\leaf\in\Partition$.
\end{enumerate}
\end{lemma}
\begin{proof}
Let $\pr:\stripSurf\to\stripSurf/\Partition$ be the quotient map onto the space of leaves.
Then the assumption on $\delta$ mean that the map $\pr\circ\delta:[0,1]\to\stripSurf/\Partition$ is locally injective and satisfies $\pr\circ\delta(0)=\pr\circ\delta(1)$.
Hence $\pr\circ\delta$ induces a continuous map of the circle $\alpha:S^1 = [0,1]/\{0,1\} \to \stripSurf/\Partition$.

\ref{enum:lm:cross_section_meets_a_singular_leaf:cylinder}
Suppose $\stripSurf$ is either a standard cylinder or a standard M\"obius band, so the space of leaves $\stripSurf/\Partition$ is homeomorphic with the unit circle $S^1$.
Hence $\alpha:S^1\to S^1$ is a continuous locally injective map between circles, and so it must be surjective.
The latter means that $\delta$ intersect each leaf of $\Partition$, i.e. $\stripSurf=\Sat{\delta([0,1])}$.

\smallskip

\ref{enum:lm:cross_section_meets_a_singular_leaf:intersect_sing_leaf}
Suppose $\stripSurf$ is neither a standard cylinder nor a standard M\"obius band.
Then by~\cite[Theorem~3.7]{MaksymenkoPolulyakh:PGC:2015} $\stripSurf$ admits a reduced atlas $\qmap: \bigsqcup \limits_{\stInd \in \StInd} \strip_{\stInd} \to \stripSurf$, that is an atlas having no leaves of types \typeCycle\ and \typeReduce.
Let $D= \qmap\bigl(\bigsqcup \limits_{\stInd \in \StInd} \partial\strip_{\stInd}\bigr)$ and $\singLeaves$ be the family of all singular leaves of $\Partition$.
As $\partial\stripSurf=\varnothing$, we get from Corollary~\ref{cor:families_of_leaves} that $D = \partial\stripSurf \cup \singLeaves = \singLeaves$.

Hence we should prove that the image of $\delta$ intersects $D$.

Suppose $\delta([0,1])\cap D = \varnothing$.
Then $\delta([0,1]) \subset \qmap(\Int{\strip_{\stInd}})$ for some $\stInd\in\StInd$.
Consider the following composition of maps:
\[
\beta = \pi \circ \qmap^{-1} \circ \delta: [0,1] \xrightarrow{~\delta~}
\qmap(\Int{\strip_{\stInd}}) \xrightarrow{~\qmap^{-1}~}
\Int{\strip_{\stInd}} = \bR\times(u,v) \xrightarrow{~\pi~} (u,v),
\]
where $\pi$ is the projection onto the second coordinate.

Then $\beta:[0,1]\to(0,1)$ is a locally injective (that is strictly monotone) continuous map satisfying $\beta(0)=\beta(1)$ which is impossible.
Hence $\delta([0,1])$ intersects $D=\singLeaves$.
\end{proof}

Now we can prove Theorem~\ref{th:charact_stripedsurf_old_ext}.
One can assume that $\stripSurf$ is connected and $\partial\stripSurf\not=\varnothing$.

Let $\specLeaves$ be the family of all special leaves of $\Partition$.
Then, by Lemma~\ref{lm:props_of_regular_leaves}, $\specLeaves \subset \singLeaves$, whence $\specLeaves$ is locally finite as well.
Moreover, by the assumption each leaf in $\Int{\stripSurf}$ satisfies condition~\ref{enum:th:charact_stripedsurf_old:eq:cross_sect} of Theorem~\ref{th:cross_sections}.
Hence by Theorem~\ref{th:charact_stripedsurf_old} the foliated surface $(\Int{\stripSurf}, \Partition_{\Int{\stripSurf}})$ admits a striped atlas.

Suppose there exists a leaf $\leaf$ of $\Partition$ belonging to $\partial\stripSurf$ and having no cross sections.
Then $\leaf$ is singular.
We will find a sequence of singular leaves converging to $\leaf$.
This will give a contradiction with the assumption that $\singLeaves$ is locally finite.

Let $x\in\leaf$, $\phi:(-1,1)\times[0,1) \to \stripSurf$ be a foliated local chart at $x$ such that $\phi(0,0)=x$, and $\delta:[0,1) \to \stripSurf$ be defined by $\delta(t) = \phi(0,t)$.

Since $\leaf$ does not admit cross sections, it follows that for each $\eps>0$ the curve $\delta((0,\eps)) \subset \Int{\stripSurf}$ intersects some leaf of $\Partition$ more that once.
So, one can find $a_{\eps} < b_{\eps}\in (0,\eps)$ such that $\delta(a_{\eps})$ and $\delta(b_{\eps})$ belong to the same leaf.
As each leaf in $\Int{\stripSurf}$ admits a cross section, it follows that $\delta((0,\eps))$ is a local cross section.
Hence, $\delta:[a_{\eps}, b_{\eps}] \to \Int{\stripSurf}$ is also a local cross section.

First suppose $\Int{\stripSurf}$ is either the standard cylinder or a M\"obius band.
Then by \ref{enum:lm:cross_section_meets_a_singular_leaf:cylinder} of Lemma~\ref{lm:cross_section_meets_a_singular_leaf} $\Sat{\delta([a_{\eps}, b_{\eps}])} = \Int{\stripSurf}$, that is the composition
\[
p\circ \delta:[a_{\eps}, b_{\eps}] \xrightarrow{~\delta~} \Int{\stripSurf} \xrightarrow{~p~} \Int{\stripSurf}/\Partition_{\Int{\stripSurf}}= S^1
\]
is surjective.
But this is will contradict to continuity of $p\circ \delta$ when $\eps \to 0$.

Therefore $\Int{\stripSurf}$ is neither the standard cylinder nor a M\"obius band.
Hence by Lemma~\ref{lm:cross_section_meets_a_singular_leaf} there exists a point $c_{\eps}\in[a_{\eps}, b_{\eps}]$ belonging to a certain singular leaf $\leaf_{\eps}$.

This implies that arbitrary small neighbourhood of $x$ intersects infinitely many singular leaves, whence $\singLeaves$ is not locally finite which contradicts to the assumption.

\providecommand{\bysame}{\leavevmode\hbox to3em{\hrulefill}\thinspace}
\providecommand{\MR}{\relax\ifhmode\unskip\space\fi MR }
\providecommand{\MRhref}[2]{%
  \href{http://www.ams.org/mathscinet-getitem?mr=#1}{#2}
}
\providecommand{\href}[2]{#2}

\end{document}